\documentclass[a4paper,12pt]{amsart}
\usepackage{amsmath,amsthm,amssymb,url}
\usepackage[stable]{footmisc}
\usepackage[latin1]{inputenc}
\usepackage[T1]{fontenc}
\usepackage[all]{xy}
\usepackage{mathrsfs}
\usepackage{hyperref}
\usepackage[usenames]{color}

\numberwithin{equation}{section}
\newtheorem{thm}[equation]{Theorem}
\newtheorem*{thm*}{Theorem}

\newtheorem{prop}[equation]{Proposition}

\theoremstyle{definition}

\newtheorem{rem}[equation]{Remark}
\newtheorem{hyp}[equation]{Hypothesis}

\usepackage{bm}
\usepackage{upgreek}
\newcommand{\bfgreek}[1]{\bm{\@nameuse{up#1}}}

\newcommand{\one}{{\mathbf{1}}}
\newcommand{\R}{\mathbb R}

\newcommand\Z{\mathbb Z}
\newcommand\A{\mathbb A}
\newcommand\Q{\mathbb Q}
\newcommand{\aaa}{\mathfrak{a}}
\newcommand{\Lie}{\operatorname{Lie}}
\newcommand\SSS{\mathcal{S}}
\newcommand\C{\mathbb C}
\newcommand\F{\mathbb F}
\newcommand\E{\mathcal E}

\newcommand\ira{\stackrel{\sim}{\rightarrow}}
\newcommand\hra{\hookrightarrow}

\newcommand\ra{\rightarrow}
\newcommand\g{\mathfrak g}
\newcommand\m{\mathfrak m}

\newcommand{\Whit}{\mathcal{W}}
\newcommand{\whit}{W}
\newcommand{\TW}{\mathcal{T}}
\newcommand{\strc}{\mathbb{S}}
\newcommand{\cusp}{\operatorname{cusp}}
\newcommand{\pair}{\mathbb{K}}
\newcommand{\per}{\mathbb{L}}
\newcommand{\coper}{\mathbb{M}}
\newcommand{\loc}{\operatorname{local}}
\newcommand{\glo}{\operatorname{global}}
\newcommand{\coh}{\operatorname{coh}}
\newcommand{\gper}{\operatorname{per}}
\newcommand{\mira}{\mathcal{P}}
\newcommand\OO{\mathcal O}
\newcommand\kk{\mathfrak k}
\newcommand\p{\mathfrak p}
\newcommand{\aut}{\operatorname{aut}}
\newcommand{\prng}{\mathbb{P}}
\newcommand{\Ad}{\operatorname{Ad}}
\newcommand{\SO}{\operatorname{SO}}
\newcommand\regulator{CRS}

\newcommand{\abs}[1]{\left|{#1}\right|}
\newcommand{\vol}{\operatorname{vol}}

\newcommand{\GL}{\operatorname{GL}}

\newcommand{\Aut}{\operatorname{Aut}}
\newcommand{\Hom}{\operatorname{Hom}}
\newcommand{\As}{\operatorname{As}}
\newcommand{\Pet}{\operatorname{Pet}}
\newcommand{\Res}{\operatorname{Res}}
\newcommand{\bs}{\backslash}

\newcommand{\swrz}{\mathcal{S}}

\title[]%
{Whittaker rational structures and special values\\ of the Asai $L$-function}
\author{Harald Grobner, Michael Harris and Erez Lapid}
\address{Harald Grobner: Fakult\"at f\"ur Mathematik\\ Universit\"at Wien\\ Oskar--Morgenstern--Platz 1\\ A-1090 Wien\\Austria}
\email{harald.grobner@univie.ac.at}

\address{Michael Harris: Univ Paris Diderot, Sorbonne Paris Cit\'e, UMR 7586, Institut de
Math\'ematiques de Jussieu-Paris Rive Gauche, Case 247, 4 place Jussieu
F-75005, Paris, France;~\\
Sorbonne Universit\'es, UPMC Univ Paris 06, UMR 7586, IMJ-PRG, F-75005
Paris, France;~\\
CNRS, UMR7586, IMJ-PRG, F-75013 Paris, France;~  \\ Department of
Mathematics, Columbia University, New York, NY  10027, USA. }
\email{}

\address{Erez Lapid: Department of Mathematics, the Weizmann Institute of Science, Rehovot 7610001, Israel}
\email{erez.m.lapid@gmail.com}

\thanks{H.G. is supported by the Austrian Science Fund (FWF), project number P 25974-N25.}
\thanks{The research leading to these results has received funding from the European Research Council under the
European Community's Seventh Framework Programme (FP7/2007-2013) / ERC Grant agreement n. 290766 (AAMOT)}
\thanks{E.L. partially supported by the ISF center of excellence grant \#1691/10 and grant \#711733 from the Minerva foundation}

\begin{document}

\maketitle

\centerline{{\it Dedicated to Jim Cogdell on the occasion of his 60th birthday}}


\begin{abstract}
Let $F$ be a totally real number field and $E/F$ a totally imaginary quadratic extension of $F$.
Let $\Pi$ be a cohomological, conjugate self-dual cuspidal automorphic representation of $\GL_n(\A_E)$.
Under a certain non-vanishing condition we relate the residue and the value of the Asai $L$-functions at $s=1$
with rational structures obtained from the cohomologies in top and bottom degrees via the Whittaker coefficient map.
This generalizes a result in Eric Urban's thesis when $n = 2$, as well as a result of the first two named authors,
both in the case $F = \Q$.
\end{abstract}

\setcounter{tocdepth}{1}
\tableofcontents

\section{Introduction}

Let $F$ be a totally real number field and $E/F$ a totally imaginary quadratic extension of $F$ with non-trivial Galois involution $\tau$.
Let $\Pi$ be a cuspidal automorphic representation of $\GL_n(\A_E)$.
One can associate two Asai $L$-functions over $F$, denoted $L(s,\Pi,\As^+)$ and $L(s,\Pi,\As^-)$.
These are Langlands $L$-functions attached to representations of the $L$-group of $\GL_n/E$,
and the Rankin-Selberg product of $\Pi$ with $\Pi^\tau$ factors as
\begin{equation}\label{RS}
L(s,\Pi \times\Pi^\tau) = L(s,\Pi,\As^+)\cdot L(s,\Pi,\As^-).
\end{equation}
In this paper we consider representations $\Pi$ that arise by stable quadratic base change from an automorphic representation $\pi$ of the unitary group $H$ over $F$.
In particular, $\Pi$ is \emph{conjugate self-dual}:
$$\Pi^{\vee} \cong \Pi^\tau.$$
We have an equality of partial $L$-functions
\[
L^S(s,\Pi,\As^{(-1)^n}) = L^S(s,\pi,\Ad),
\]
where $\Ad$ is the adjoint representation of the $L$-group of $H$.
(The other Asai $L$-function equals the $L$-function of $\pi$ with respect to the twist of $\Ad$ by the character corresponding to $E/F$.)


Since $\Pi^\tau \cong \Pi^{\vee}$, the Rankin-Selberg $L$-function on the left-hand side of \eqref{RS} has a simple pole at $s = 1$ and the assumption that
$\Pi$ is a base change from a unitary group implies that this pole arises as the pole of $L(s,\Pi,\As^{(-1)^{n-1}})$ at $s = 1$. Moreover, $L(s,\Pi,\As^{(-1)^n})$ is holomorphic and non-vanishing at $s=1$. This applies in particular if $\Pi$ is cohomological and conjugate-dual: Then it is known that $\Pi$ is automatically a base change from some unitary group $H$, and moreover $s = 1$ is a critical value of $L(s,\Pi,\As^{(-1)^n})$.

Hypothetically,
$L(s,\Pi,\As^{(-1)^{n-1}}) = \zeta(s)L(s,M^{\flat}(\Pi))$ for some motive $M^{\flat}(\Pi)$ (which we do not specify here), where $\zeta(s)$ is the Riemann zeta function.
One of the main goals of this note is to relate the residue at $s = 1$ of $L(s,\Pi,\As^{(-1)^{n-1}})$, which under the above hypothesis can be interpreted as a \emph{non-critical}
special value of the $L$-function of $M^{\flat}(\Pi)$, to a certain cohomology class attached to $\Pi$,
of the adelic ``locally symmetric'' space $\SSS_E=\GL_n(E)\bs \GL_n(\A_E)/A_G K_\infty$.

In fact, $\Pi$ contributes to the cohomology of $\SSS_E$ (with suitable coefficients) in several degrees.
For each degree $q$, where it contributes, one can define a rational structure on the $q$-th $(\m_G,K_\infty)$-cohomology of $\Pi_\infty$, which measures
the difference between the global cohomological rational structure and the one defined using the Whittaker-Fourier coefficient.
We call it the {\it Whittaker comparison rational structure} (\regulator) of degree $q$ and denote it by $\strc_q$.
Let $b$ and $t$ be the minimal and maximal degrees, respectively, where the $(\m_G,K_\infty)$-cohomology of $\Pi_\infty$ is non-zero.
Roughly speaking, our main result is that under a suitable local non-vanishing assumption,
$\Res_{s = 1}L(s,\Pi,\As^{(-1)^{n-1}})$ (resp., $L(1,\Pi,\As^{(-1)^n})$) spans, under suitable normalization,
the one-dimensional spaces $\strc_t$ (resp., $\strc_b$) over the field of definition of $\Pi$.
The precise results are stated in Theorems \ref{thm:Asai} and \ref{thm: crit} in the body of the paper.
In the case $n=2$ and $F=\Q$ such results had been proved in the theses of Eric Urban and Eknath Ghate, respectively \cite{MR1361742, MR1671216}.
The results here sharpen some of the main results of \cite{1308.5090} (for $F = \Q$).
We are hopeful that the pertinent non-vanishing assumption will be settled in the near future using the method recently developed by Binyong Sun.

The result for the top degree cohomology turns out to be a rather direct consequence of the well-known relation between the residue
of the Asai $L$-function at $s=1$ and the period integral over $\GL_n(F)\bs\GL_n(\A)/A_G$ \cite{MR984899}.
This is a twisted analogue of the realization of the Petersson inner product in the Whittaker model, due to Jacquet--Shalika \cite{MR623137}.

The two results for the top and bottom degrees are linked by Poincar\'e duality and the relation \eqref{RS}.
More precisely, one can relate $\Res_{s=1}L(s,\Pi\times\Pi^\vee)$
to a suitable pairing between $\strc_q$ and $\strc_{d-q}$ in any degree $q$ (where $d$ is the dimension of $\SSS_E$).
Once again, this is a simple consequence of the aforementioned result of Jacquet--Shalika.
Although we prove this Poincar\'e duality result only for (cohomological) conjugate self-dual representations of $\GL_n(\A_E)$,
the methodology is applicable for any cohomological representation over any number field.
In fact, a result of a similar flavor (in this generality) was obtained independently about the same time by Balasubramanyam--Raghuram \cite{1408.3896}.

We also note that A. Venkatesh has recently proposed a conjecture relating the rational structure of the contribution of
$\Pi$ to cohomology (in all degrees) to the $K$-theoretic regulator map of a hypothetical motive attached to $\Pi$.
We hope that the precise statement of the conjecture will be available soon.
At any rate, our results seem to be compatible with this conjecture.

We thank the anonymous referee for a very careful reading of the manuscript and for pointing out several inaccuracies in the
original version of the paper.

\section{Notation and conventions}

\subsection{Number fields and the groups under consideration}\label{sect:numberfields}
Throughout the paper, let $F$ be a totally real number field and $E/F$ denotes a totally imaginary quadratic extension of $F$
with non-trivial Galois involution $\tau$. The discriminant of $F$ (resp.\ $E$) is denoted $D_F$ (resp.\ $D_E$).
We let $I_E$ be the set of field embeddings $E\hra\C$ and the ring of integers (resp., adeles) of $E$ by $\OO_E$ and $\A_E$.
(Similar notation is used for the field $F$.)
We fix a non-trivial, continuous, additive character $\psi:E\bs \A_E\ra\C^\times$.

We fix an integer $n\ge 1$ and use $G$ to denote the general linear group $\GL_n$ viewed as a group scheme over $\Z$.

We denote by $T$ the diagonal torus in $G$, by $\mira$ the subgroup of $G$ consisting of matrices whose last row
is $(0,\dots,0,1)$, and by $U$ the unipotent subgroup of upper triangular matrices in $G$.

For brevity we write $G_\infty=R_{E/\Q}(G)(\R)$, where $R_{E/\Q}$ stands for the restriction of scalars from $E$ to $\Q$. We write $A_G$ for the group of positive reals $\R_+$ embedded diagonally in the center of $G_\infty$. Thus, $G(\A_E)\cong G(\A_E)^1\times A_G$ where $G(\A_E)^1:=\{g\in G(\A_E):\abs{\det g}_{\A_E^*}=1\}$.

Let $K_\infty$ be the standard maximal compact subgroup of $G_\infty$ isomorphic to $U(n)^{[F:\Q]}$.
We set $\g_\infty=\Lie(G_\infty)$, $\kk_\infty=\Lie(K_\infty)$, $\aaa_G=\Lie(A_G)$, $\p=\g_\infty/\kk_\infty$,
$\m_G=\g_\infty/\aaa_G$ and $\tilde\p=\p/\aaa_G$. Let $d:=\dim_\R\m_G -\dim_\R\kk_\infty$.
The choice of measures is all-important in all results of this kind.  Our choices are specified in Sections
\ref{sec:  measures} and \ref{sec: measures_F}.

\subsection{Coefficient systems}\label{sect:coeffsys}
We fix an irreducible, finite-dimensional, complex, continuous algebraic representation $E_\mu$ of $G_{\infty}$.
It is determined by its highest weight $\mu=(\mu_\iota)_{\iota\in I_E}$ where for each $\iota$,
$\mu_\iota=(\mu_{1,\iota},...,\mu_{n,\iota})\in\Z^n$ with $\mu_{1,\iota}\geq\mu_{2,\iota}\geq...\geq\mu_{n,\iota}$.
We assume that $E_\mu$ is conjugate self-dual, i.e., $E_\mu^\tau\cong E_\mu^\vee$, or, in other  words, that
$$\mu_{j,\tau(\iota)}+\mu_{n+1-j,\iota}=0, \quad\quad \iota\in I_E,\ 1\leq j\leq n.$$

\subsection{Cuspidal automorphic representations}
Let $\Pi$ be a cuspidal automorphic representation of $G(\A_E)=\GL_n(\A_E)$.
We shall assume that $\Pi$ is conjugate self-dual, i.e., $\Pi^\tau\cong\Pi^\vee$.
We say that $\Pi$ is \emph{cohomological} with respect to $E_\mu$, if $H^*(\g_\infty,K_\infty,\Pi\otimes E_\mu)\neq 0$.
We refer to Borel--Wallach \cite{MR1721403}, I.5, for details concerning $(\g_\infty,K_\infty)$-cohomology.

{\it Throughout the paper we assume that $\Pi$ is a conjugate self-dual, cuspidal automorphic representation
which is cohomological with respect to $E_\mu$.}  (However, this hypothesis is not used before the proof of
Theorem \ref{thm:Whittaker}.)
Denote the Petersson inner product on $\Pi\times\Pi^\vee$ by
\[
(\varphi,\varphi^\vee)_{\Pet}:=\int_{G(E)\bs G(\A_E)^1}\varphi(g)\varphi^\vee(g)\ dg.
\]
We write $\Whit^{\psi}(\Pi)$ for the Whittaker model of $\Pi$ with respect to the character
\[
\psi_U(u)=\psi(u_{1,2}+\dots+u_{n-1,n}).
\]
Similarly for $\Whit^{\psi^{-1}}(\Pi^\vee)$. Let
\[
\whit^{\psi}:\Pi\rightarrow\Whit^{\psi}(\Pi)
\]
be the realization of $\Pi$ in the Whittaker model via the $\psi$-Fourier coefficient, namely
\[
\whit^{\psi}(\varphi)=(\vol(U(E)\bs U(\A_E)))^{-1}\int_{U(E)\bs U(\A_E)}\varphi(ug)\psi(u)^{-1}\ du.
\]
Analogous notation is used locally.

\subsection{Pairings of $(\m_G,K_\infty)$-cohomology spaces}\label{sect:pairing_gK}
Suppose that $\rho$ and $\rho^*$ are two irreducible $(\g_\infty,K_\infty)$-modules which are in duality and let $(\cdot,\cdot)$ be
a non-degenerate invariant pairing on $\rho\times\rho^*$.
For $p+q=d$, let us define a pairing
$$\pair_{(\rho,\rho^*,(\cdot,\cdot))}^{\coh,p}: H^p(\m_G,K_\infty,\rho\otimes E_\mu)\times H^q(\m_G,K_\infty,\rho^*\otimes E^\vee_\mu)\ra (\wedge^d\tilde\p)^*$$
as follows: Recall that
\begin{gather*}
 H^p(\m_G,K_\infty,\rho\otimes E_\mu)\ira\Hom_{K_\infty}(\wedge^p\tilde\p,\rho\otimes E_\mu),\\
 H^q(\m_G,K_\infty,\rho^*\otimes E_\mu^\vee)\ira\Hom_{K_\infty}(\wedge^q\tilde\p,\rho^*\otimes E_\mu^\vee).
\end{gather*}
Suppose that $\tilde\omega\in \Hom_{K_\infty}(\wedge^p\tilde\p,\rho\otimes E_\mu)$ and
$\tilde\eta\in \Hom_{K_\infty}(\wedge^q\tilde\p,\rho^*\otimes E_\mu^\vee)$ represent
$\omega\in H^p(\m_G,K_\infty,\rho\otimes E_\mu)$ and\
$\eta\in H^q(\m_G,K_\infty,\rho^*\otimes E_\mu^\vee)$ respectively.
The cap product
\[
 \tilde\omega\wedge\tilde\eta\in \Hom_{K_\infty}(\wedge^d\tilde\p,\rho\otimes\rho^*\otimes E_\mu\otimes E_\mu^\vee)
\]
together with the pairing on $\rho\times\rho^*$ and the canonical pairing on $E_\mu\otimes E_\mu^\vee$, define an element
\[
\pair_{(\rho,\rho^*,(\cdot,\cdot))}^{\coh,p}(\omega,\eta)\in (\wedge^d\tilde\p)^*.
\]
Note that $(\wedge^d\p)^*$ is canonically isomorphic to the space of invariant measures on
$G_\infty/A_G K_\infty$.

\subsection{Locally symmetric spaces over $E$}
Recall the adelic quotient
$$\SSS_E:=G(E)\bs G(\A_E)^1/K_\infty.$$
We can view $\SSS_E$ as the projective limit $\SSS_E=\varprojlim_{K_f} \SSS_{E,K_f}$
where
\[
 \SSS_{E,K_f}=G(E)\bs G(\A_E)^1/K_\infty K_f
\]
and $K_f$ varies over the directed set of compact open subgroups of $G(\A_{E,f})$ ordered by opposite inclusion.
Note that each $\SSS_{E,K_f}$ is a orbifold of dimension $d=n^2[F:\Q]-1$.

A representation $E_\mu$ as in \S\ref{sect:coeffsys} defines a locally constant sheaf $\E_\mu$ on $\SSS_E$
whose espace \'etal\'e is $G(\A_E)^1/K_\infty\times_{G(E)} E_\mu$ with the discrete topology on $E_\mu$.
We denote by $H^q(\SSS_E,\E_\mu)$ and $H^q_c(\SSS_E,\E_\mu)$ the corresponding spaces of sheaf cohomology and sheaf
cohomology with compact support, respectively. They are $G(\A_{E,f})$-modules.
We have
\[
H^q(\SSS_E,\E_\mu)\cong \varinjlim_{K_f} H^q(\SSS_{E,K_f},\E_\mu)
\]
and
\[
 H_c^q(\SSS_E,\E_\mu)\cong \varinjlim_{K_f} H_c^q(\SSS_{E,K_f},\E_\mu),
\]
where the maps in the inductive systems are the pull-backs (Rohlfs \cite{MR1414393} Cor.\ 2.12 and Cor.\ 2.13).
For our purposes we will only use this result to save notation (or to avoid an abuse of notation): we could have simply worked throughout with the inductive limits of cohomologies. (In fact, in \cite{MR1044819} $H^q(\SSS_E,\E_\mu)$ is simply defined as $\varinjlim_{K_f} H^q(\SSS_{E,K_f},\E_\mu)$.)

Let $H^q_{\cusp}(\SSS_E,\E_\mu)$ be the $G(\A_{E,f})$-module of cuspidal cohomology, being defined as the $(\m_G,K_\infty)$-cohomology of the space of cuspidal automorphic forms. As cusp forms are rapidly decreasing, we obtain an injection
\[
\Delta^q: H^q_{\cusp}(\SSS_E,\E_\mu)\hookrightarrow H_c^q(\SSS_E,\E_\mu).
\]

\section{Instances of algebraicity}
\subsection{An action of $\Aut(\C)$}
Let $\nu$ be a smooth representation of either $G(\A_{E,f})$ or $G(E_w)$ for a non-archimedean place $w$ of $E$, on a complex vector space $W$.
For $\sigma\in \Aut(\C)$, we define the \emph{$\sigma$-twist} ${}^\sigma\!\nu$ following Waldspurger \cite{MR783510}, I.1:
If $W'$ is a $\C$-vector space which admits a $\sigma$-linear isomorphism $\phi:W\ra W'$ then we set
$${}^\sigma\!\nu:= \phi\circ\nu\circ \phi^{-1}.$$
This definition is independent of $\phi$ and $W'$ up to equivalence of representations.
One may hence always take $W':=W\otimes_{\sigma}\C$. \\\\
On the other hand, let $\nu=E_\mu$ be a highest weight representation of $G_\infty$ as in \S \ref{sect:coeffsys}.
The group $\Aut(\C)$ acts on $I_E$ by composition. Hence, we may define ${}^\sigma\!E_\mu$ to be the irreducible representation of $G_\infty$,
whose local factor at the embedding $\iota$ is $E_{\mu_{\sigma^{-1}\iota}}$, i.e., has highest weight $\mu_{\sigma^{-1}\iota}$.
As a representation of the diagonally embedded group $G(E)\hra G_\infty$, ${}^\sigma\!E_\mu$ is isomorphic to $E_\mu\otimes_{\sigma}\C$,
cf.\ Clozel \cite{MR1044819}, p.\ 128. Moreover, we obtain

\begin{prop}\label{prop:SigmaPi}
For all $\sigma\in\Aut(\C)$, ${}^\sigma\Pi_f$ is the finite part of a cuspidal automorphic representation ${}^\sigma\Pi$
which is cohomological with respect to ${}^\sigma\!E_\mu$. The representation ${}^\sigma\Pi$ is conjugate self-dual.
\end{prop}

\begin{proof}
See \cite{MR1044819}, Thm.\ 3.13.
(Note that an irreducible $(\g_\infty,K_\infty)$-module is cohomological if and only if it is regular algebraic in the sense of [loc. cit.].)
The last statement is obvious.
\end{proof}

\subsection{Rationality fields and rational structures}
Recall also the definition of the rationality field of a representation (e.g., \cite{MR783510}, I.1).
If $\nu$ is any of the representations considered above, let $\mathfrak S(\nu)$ be the group of all automorphisms
$\sigma\in \Aut(\C)$ such that ${}^\sigma\!\nu\cong\nu $:
$$\mathfrak S(\nu):=\{\sigma\in \Aut(\C)|{}^\sigma\!\nu\cong\nu\}.$$
Then the \emph{rationality field} $\Q(\nu)$ is defined as the fixed field of $\mathfrak S(\nu)$,
$$\Q(\nu):=\{z\in\C| \sigma(z)=z \textrm{ for all } \sigma\in\mathfrak S(\nu)\}.$$
As a last ingredient we recall that a group representation $\nu$ on a $\C$-vector space $W$ is said to be
\emph{defined over a subfield $\F\subset\C$}, if there exists an $\F$-vector subspace $W_\F\subset W$, stable under the group action,
and such that the canonical map $W_\F\otimes_\F\C\rightarrow W$ is an isomorphism. In this case, we say that $W_\F$
is an \emph{$\F$-structure} for $(\nu,W)$.

\begin{rem} \label{eq: tprs}
If $(\nu,W)$ is irreducible, then a rational structure is unique up to homothety, if it exists. Moreover, if $W_\F$ is an $\F$-structure for $(\nu,W)$,
with $(\nu,W)$ irreducible, and if $V$ is a complex vector space with a trivial group action then any $\F$-structure for $(\nu\otimes 1,W\otimes V)$ is of the form
$W_\F\otimes V_\F$ for a unique $\F$-structure $V_\F$ of $V$ (as a complex vector space).
\end{rem}

It is easy to see that as a representation of $G(E)$, $E_\mu$ has a $\Q(E_\mu)$-structure, whence, so does $H^q(\SSS_E,\E_\mu)$, cf.\ \cite{MR1044819}, p.\ 122.

\begin{prop}\label{prop:PiStructure}
Let $\Pi$ be a cuspidal automorphic representation of $G(\A_E)$. Then $\Pi_f$ has a $\Q(\Pi_f)$-structure, which is unique up to homotheties.
If $\Pi$ is cohomological with respect to $E_\mu$, then $\Q(\Pi_f)$ is a number field.
Similarly, $H^q(\m_G,K_\infty,\Pi\otimes E_\mu)$ has a $\Q(\Pi_f)$-structure coming from the natural $\Q(E_\mu)$-structure of
$H^q(\SSS_E,\E_\mu)$.
\end{prop}

\begin{proof}
This is contained in \cite{MR1044819}, Prop.\ 3.1, Thm.\ 3.13 and Prop.\ 3.16 (the Drinfeld-Manin principle).
The reader may also have a look at \cite{MR3208871} Thm.\ 8.1 and Thm.\ 8.6.
For the last statement, one observes that $\Q(E_\mu)\subseteq\Q(\Pi_f)$ by Strong Multiplicity One (Cf.\ \cite{MR3208871}, proof of Cor.\ 8.7.).
\end{proof}

\section{The Whittaker \regulator s}

\subsection{Rational structures on Whittaker models} \label{sec: ratwhit}
We recall the discussion of \cite{MR2439563}, \S 3.2, resp.\ \cite{MR2171731}, \S 3.3. Fix a non-archimedean place $w$ of $E$. Given a Whittaker function $\xi$ on $G(E_w)$ and $\sigma\in\Aut(\C)$ we define the Whittaker function ${}^{\sigma}\xi$ by
\begin{equation}
\label{eqn:aut-c-action}
{}^{\sigma}\xi(g): = \sigma(\xi(t_\sigma\cdot g)),
\end{equation}
where $t_\sigma$ is the (unique) element in $T(E_w)\cap\mira(E_w)$ which conjugates $\psi_U$ to $\sigma\psi_U$.
Note that $t_\sigma$ does not depend on $\psi$. We have, $t_{\sigma_1\sigma_2}=t_{\sigma_1}t_{\sigma_2}$
and hence ${}^{\sigma_1\sigma_2}\xi\equiv{}^{\sigma_1}({}^{\sigma_2}\xi)$ for all $\sigma_1,\sigma_2\in\Aut(\C)$.
Thus, if $\pi$ is any irreducible admissible generic representation of $G(E_w)$, then we obtain a $\sigma$-linear intertwining operator
$\TW_\sigma^\psi:\Whit^{\psi}(\pi)\rightarrow\Whit^{\psi}({}^\sigma\pi)$.
In particular, we get a $\Q(\pi)$ structure on $\Whit^{\psi}(\pi)$ by taking invariant vectors.
A similar discussion applies to irreducible admissible generic representations of $G(\A_{E,f})$.

\subsection{}
The map $\whit^\psi:\Pi\rightarrow\Whit^\psi(\Pi)$ gives rise to an isomorphism
\begin{multline} \label{eq: whipsiHi}
H^q(\m_G,K_\infty,\Pi\otimes E_\mu)\stackrel{\sim}{\longrightarrow}\\
H^q(\m_G,K_\infty,\Whit^{\psi}(\Pi)\otimes E_\mu)\cong
H^q(\m_G,K_\infty,\Whit^{\psi_\infty}(\Pi_\infty)\otimes E_\mu)\otimes\Whit^{\psi_f}(\Pi_f).
\end{multline}
Recall the $\Q(\Pi_f)$-structure on $H^q(\m_G,K_\infty,\Pi\otimes E_\mu)$, (respectively on $\Whit^{\psi_f}(\Pi_f)$) from Prop.\ \ref{prop:PiStructure} (respectively from \S\ref{sec: ratwhit}). Thus, by Rem.\ \eqref{eq: tprs} we obtain a $\Q(\Pi_f)$-structure on the cohomology space
$H^q(\m_G,K_\infty,\Whit^{\psi_\infty}(\Pi_\infty)\otimes E_\mu)$ (as a $\C$-vector space)
which we denote by $\strc^{\psi,q}_{\Pi}$ and call it the \emph{$q$-th Whittaker comparison rational structure} (\regulator) of $\Pi$.
In particular, $\strc^{\psi,q}_{\Pi}$ is a $\Q(\Pi_f)$-vector subspace of $H^q(\m_G,K_\infty,\Whit^{\psi_\infty}(\Pi_\infty)\otimes E_\mu)$
and $\dim_{\Q(\Pi_f)} \strc^{\psi,q}_{\Pi}=\dim_\C H^q(\m_G,K_\infty,\Whit^{\psi_\infty}(\Pi_\infty)\otimes E_\mu)$.

\subsection{Equivariance of local Rankin-Selberg integrals}\footnote{Essentially the same argument is given in the
middle of the proof of Theorem 2 of \cite{MR1969002}. We have restated it separately for convenience.}
\label{eq: equiRS}

For this subsection only let $F$ be a local non-archimedean field.
Let $\pi_1$ and $\pi_2$ be two generic irreducible representations of $\GL_n(F)$
with Whittaker models $\Whit^{\psi}(\pi_1)$ and $\Whit^{\psi^{-1}}(\pi_2)$ respectively.
Given $W_1\in\Whit^{\psi}(\pi_1)$, $W_2\in\Whit^{\psi^{-1}}(\pi_2)$ and $\Phi\in\swrz(F^n)$ a Schwartz-Bruhat function, consider the zeta integral
\[
Z^\psi(W_1,W_2,\Phi,s)=\int_{U(F)\bs\GL_n(F)}W_1(g)W_2(g)\Phi(e_ng)\abs{\det g}^s\ dg
\]
where $e_n$ is the row vector $(0,\dots,0,1)$ and the invariant measure $dg$ on $U(F)\bs\GL_n(F)$ is \emph{rational}, i.e.,
it assigns rational numbers to compact open subsets.
We view the above expression as a formal Laurent series $A^\psi(W_1,W_2,\Phi)\in\C((X))$ in $X=q^{-s}$ whose $m$-th coefficient $c_m^\psi(W_1,W_2,\Phi)$ is
\[
\int_{U(F)\bs\GL_n(F):\abs{\det g}=q^{-m}}W_1(g)W_2(g)\Phi(e_ng)\ dg.
\]
The last integral reduces to a finite sum, and vanishes for $m\ll0$, because of the support of Whittaker functions.
It is therefore clear (by a simple change of variable) that
\[
\sigma(c_m^\psi(W_1,W_2,\Phi))=c_m^{\psi}({}^\sigma W_1,{}^\sigma W_2,\sigma\Phi)
\]
for any $\sigma\in\Aut(\C)$.
Thus,
\[
A^{\psi}({}^\sigma W_1,{}^\sigma W_2,\sigma\Phi)=(A^\psi(W_1,W_2,\Phi))^\sigma
\]
where $\sigma$ acts on $\C((X))$ in the obvious way. The linear span of
\[
A^\psi(W_1,W_2,\Phi),\ W_1\in\Whit^\psi(\pi_1),\ W_2\in\Whit^{\psi^{-1}}(\pi_2),\ \Phi\in\swrz(F^n)
\]
is a fractional ideal $\mathfrak{I^\psi}(\pi_1,\pi_2)$ of $\C[X,X^{-1}]$.
Thus, by the above,
\[
\mathfrak{I^\psi}(\pi_1,\pi_2)^\sigma=\mathfrak{I^\psi}({}^\sigma \pi_1,{}^\sigma \pi_2)
\]
for any $\sigma\in\Aut(\C)$. Hence, if we write $L(s,\pi_1\times\pi_2)=(P_{\pi_1,\pi_2}(q^{-s}))^{-1}$
where $(P_{\pi_1,\pi_2}(X))^{-1}$ is the generator of $\mathfrak{I^\psi}(\pi_1,\pi_2)$
such that $P_{\pi_1,\pi_2}\in\C[X]$ and $P_{\pi_1,\pi_2}(0)=1$ then it follows that
$P_{{}^\sigma \pi_1,{}^\sigma \pi_2}=P_{\pi_1,\pi_2}^\sigma$.

Of course, this argument applies equally well to other $L$-factors defined by the Rankin-Selberg method.

\section{A cohomological interpretation of $\Res_{s=1}L(s,\Pi\times\Pi^\vee)$}\label{sect:RS}

\subsection{A pairing} \label{sec: pairing}
For any compact open subgroup $K_f$ of $G(\A_{E,f})$ we use the de Rham isomorphism to define a canonical map
of vector spaces
\[
H^d_c(\SSS_{E,K_f},\underline{\C})\xrightarrow{\int_{\SSS_{E,K_f}}}\C.
\]
Thus, if $p+q=d$ then we get a canonical non-degenerate pairing
\begin{equation} \label{eq: pairK}
 H^p_c(\SSS_{E,K_f},\E_\mu)\times H^q(\SSS_{E,K_f},\E^\vee_\mu)\rightarrow\C
\end{equation}
which is defined by taking the cap product to $H^d_c(\SSS_E,\E_\mu\otimes\E^\vee_\mu)$,
mapping it to $H^d_c(\SSS_E,\underline{\C})$ using the canonical map $\E_\mu\otimes\E^\vee_\mu\rightarrow\underline{\C}$
and finally applying $\int_{\SSS_{E,K_f}}$.
Note however that, as defined, the maps $\int_{\SSS_{E,K_f}}$ do not fit together compatibly to a map $H^d_c(\SSS_E,\underline{\C})\rightarrow\C$, since we have to take into account the degrees of the covering maps $\SSS_{E,K_f}\rightarrow\SSS_{E,K'_f}$, $K_f\subset K'_f$.
To rectify the situation, we fix once and for all a $\Q$-valued Haar measure $\gamma$ on $G(\A_{E,f})$
(which is unique up to multiplication by $\Q^*$). The normalized integrals $\int'_{\SSS_{E,K_f}}:
H^d_c(\SSS_{E,K_f},\underline{\C})\rightarrow\C$
given by $\int'_{\SSS_{E,K_f}}=\vol_\gamma(K_f)\int_{\SSS_{E,K_f}}$
are compatible with the pull-back with respect to the covering maps $\SSS_{E,K_f}\rightarrow\SSS_{E,K'_f}$, $K_f\subset K'_f$.
Thus, we get a map
\[
H^d_c(\SSS_E,\underline{\C})\xrightarrow{\int'_{\SSS_E}}\C.
\]
We denote the resulting pairing
\[
\prng^p: H^p_c(\SSS_E,\E_\mu)\times H^q(\SSS_E,\E^\vee_\mu)\rightarrow\C
\]
It depends implicitly on the choice of $\gamma$, but this ambiguity is only up to an element of $\Q^*$.\footnote{This point is implicit in \cite[p. 124]{MR1044819}.} At any rate, the maps $\int_{\SSS_{E,K_f}}$ (and consequently, $\int'_{\SSS_{E,K_f}}$ and $\int'_{\SSS_E}$) are $\Aut(\C)$-equivariant with respect to the standard rational structure of $H^d_c(\SSS_{E,K_f},\underline{\C})$ and $H^d_c(\SSS_E,\underline{\C})$. The same is therefore true for the pairing \eqref{eq: pairK} and $\prng^p$ (with respect to the rational structure of $\E_\mu$).

As noted in \cite[p. 124]{MR1044819}, the pairing $\prng^p$ restricts to a non-degenerate pairing
\[
 H^p_{\cusp}(\SSS_E,\E_\mu)\times H^q_{\cusp}(\SSS_E,\E^\vee_\mu)\rightarrow\C
\]
and therefore to a non-degenerate pairing
\[
 H^p_{\Pi_f}(\SSS_E,\E_\mu)\times H^q_{\Pi^\vee_f}(\SSS_E,\E^\vee_\mu)\rightarrow\C
\]
where $H^p_{\Pi_f}(\SSS_E,\E_\mu)$ is the $\Pi_f$-isotypic part of $H^p_{\cusp}(\SSS_E,\E_\mu)$
and similarly for $H^q_{\Pi^\vee_f}(\SSS_E,\E^\vee_\mu)$. Composing with $\Delta^p$ and $\Delta^q$ we finally get a non-degenerate pairing
$$\pair^{\Pet,p}: H^p(\m_G,K_\infty,\Pi\otimes E_\mu)\times H^q(\m_G,K_\infty,\Pi^\vee\otimes E^\vee_\mu)\ra\C.$$
This pairing coincides with the volume of $\SSS_E$ with respect to the complex-valued measure
$\pair_{(\Pi_\infty,\Pi_\infty^\vee,(\cdot,\cdot)_{\Pet})}^{\coh,p}\otimes\gamma$ of $G(\A_E)^1/K_\infty\cong G_\infty/A_G K_\infty\times G(\A_{E,f})$.
Here we identify $(\wedge^d\p)^*$ with the space of invariant measures on $G_\infty/A_G K_\infty$.

\subsection{Measures over $E$} \label{sec: measures}
At this point it will be convenient to introduce some Haar measures on various groups.
If $w$ is non-archimedean we take the Haar measure on $E_w$ which gives volume one to the integers $\OO_w$.
On $\C$ we take twice the Lebesgue measure.
Having fixed measures on $E_w$ for all $w$ we can define (unnormalized) Tamagawa measures on local groups by
providing a gauge form (up to a sign).
On the groups $G$, $\mira$ and $U$ we take the gauge form $\wedge dx_{i,j}/(\det x)^k$ where $(i,j)$
range over the coordinates of the non-constant entries in the group and $k$ is $n$, $n-1$ and $0$ respectively.
Note that if $w$ is non-archimedean then the volume of $G(\OO_w)$ is $\Delta_{G,w}^{-1}$ where
$\Delta_{G,w}=\prod_{j=1}^nL(j,\one_{E_w^*})$.
On $G(\A_E)$ and $G(\A_{E,f})$ we will take the measure
\[
\prod_w\Delta_{G,w}\ dg_w
\]
where $w$ ranges over all places (resp., all finite places).
On $A_G$ we take the Haar measure whose push-forward (to $\R_+$) under $\abs{\det}_{\A_E^*}$ is $dx/x$
where $dx$ is the Lebesgue measure. The isomorphism $G(\A_E)\cong A_G\times G(\A_E)^1$
gives a measure on $G(\A_E)^1$.
Then
\[
\vol(G(E)\bs G(\A_E)^1)=\abs{D_E}^{n^2/2}\Res_{s=1}\prod_{j=1}^n\zeta_E^*(s+j-1)
\]
where $D_E$ is the discriminant of $E$ and $\zeta_E^*(s)$ is the completed Dedekind zeta function of $E$.

Let $\xi\in(\wedge^d\tilde\p)^*$ correspond to the invariant measure on $A_G\bs G_\infty/K_\infty$
obtained by the push-forward of the Haar measure on $A_G\bs G_\infty$ chosen above.
Let $\Lambda_0\in\wedge^d\tilde\p$ be the element such that $\xi(\Lambda_0)=1$.


\subsection{The Whittaker realization of the Petersson inner product}

Given a finite set of places $S$ of $E$ and an irreducible generic essentially unitarizable representation $\pi_S$ of $G(E_S)$ with Whittaker
model $\Whit^{\psi_S}(\pi_S)$ define
\[
 [W,W^\vee]_S:=\frac{\Delta_{G,S}}{L(1,\pi_S\times\pi_S^\vee)}\cdot
 \int_{U(E_S)\bs\mira(E_S)}W(g)W^\vee(g)\ dg, \ \ W\in\Whit^{\psi}(\pi_S),\ W^\vee\in\Whit^{\psi^{-1}}(\pi^\vee_S)
\]
where $\Delta_{G,S}=\prod_{w\in S}\Delta_{G,w}$.
It is well known that this integral converges and defines a $G(E_S)$-invariant pairing on $\Whit^{\psi}(\pi_S)\times\Whit^{\psi^{-1}}(\pi^\vee_S)$.
If $S$ consists of the archimedean places of $E$, we simply write $[W,W^\vee]_\infty$.

We note that if $S$ consists of non-archimedean places only, then $[W,W^\vee]_S$ is $\Aut(\C)$-equivariant, i.e.,
\[
  [{}^\sigma W,{}^\sigma W^\vee]_S= \sigma([W,W^\vee]_S).
\]
Indeed, by uniqueness, it suffices to check this relation when the restriction of $W$ to $\mira(E_S)$ is compactly supported modulo $U(E_S)$,
in which case the integral reduces to a finite sum and the assertion follows from \S\ref{eq: equiRS} and the fact that the measures
chosen on $U(E_S)$ and $\mira(E_S)$ assign rational values to compact open subgroups.

With our choice of measures, given cuspidal automorphic forms $\varphi$, $\varphi^\vee$ in the space of $\Pi$ and $\Pi^\vee$,
respectively, we abbreviate $\whit^\psi_\varphi=\whit^\psi(\varphi)$, $\whit^{\psi^{-1}}_{\varphi^\vee}=\whit^{\psi^{-1}}(\varphi^\vee)$ and obtain
\begin{equation} \label{eq: loctoglob}
(\varphi,\varphi^\vee)_{\Pet}=\abs{D_E}^{n(n+1)/4}
\Res_{s=1}L(s,\Pi\otimes\Pi^\vee)[\whit^\psi_\varphi,\whit^{\psi^{-1}}_{\varphi^\vee}]_S
\end{equation}
(see \cite[p. 477]{MR3267120} which is of course based on \cite{MR623137})
provided that $S$ is a finite set of places of $E$ containing all the archimedean ones as well as all the non-archimedean
places for which either $\varphi$ or $\varphi^\vee$ is not $G(\OO_w)$-invariant or the conductor of $\psi_w$ is different from $\OO_w$.
(Note that $[\whit^\psi_\varphi,\whit^{\psi^{-1}}_{\varphi^\vee}]_S$ is unchanged by enlarging $S$
because of the extra factor $\Delta_{G,S}$ in the numerator.)

We will also write $[W,W^\vee]_f=[W,W^\vee]_S$
for any $W\in\Whit^{\psi_f}(\Pi_f)$ and $W^\vee\in\Whit^{\psi_f^{-1}}(\Pi_f^\vee)$
where $S$ is any sufficiently large set of non-archimedean places of $E$ (depending on $W$ and $W^\vee$).

\subsection{A relation between the Whittaker \regulator s and $\Res_{s=1}L(s,\Pi\times\Pi^\vee)$}

\begin{thm}\label{thm:Whittaker}
Let $\Pi$ be a conjugate self-dual, cuspidal automorphic representation of $G(\A_E)=\GL_n(\A_E)$, which is cohomological with respect to an irreducible,
finite-dimensional, algebraic representation $E_\mu$. For all degrees $p$, the number $\left(\abs{D_E}^{n(n+1)/4}\Res_{s=1}L(s,\Pi\times\Pi^\vee)\right)^{-1}$ spans the one-dimensional $\Q(\Pi_f)$-vector space
\[
\Lambda\circ\pair_{(\Whit^{\psi_\infty}(\Pi_\infty),\Whit^{\psi_\infty^{-1}}(\Pi_\infty^\vee),[\cdot,\cdot]_\infty)}^{\coh,p}
(\strc^{\psi,p}_\Pi,\strc^{\psi^{-1}, d-p}_{\Pi^\vee})\subset\C
\]
where $\Lambda:(\wedge^d\tilde\p)^*\rightarrow\C$ is the evaluation at the element $\Lambda_0\in\wedge^d\tilde\p$ defined in \S\ref{sec: measures}.
\end{thm}

\begin{proof}
We have two pairings $\pair^{\loc,p}_{\Pi_\infty}$ and $\pair_{\Pi}^{\glo,p}$ on
\[
H^p(\m_G,K_\infty,\Whit^{\psi_\infty}(\Pi_\infty)\otimes E_\mu)\times
H^q(\m_G,K_\infty,\Whit^{\psi_\infty}(\Pi_\infty^\vee)\otimes E_\mu^\vee)
\]
(with $p+q=d$): Firstly, the local pairing $\pair^{\loc,p}_{\Pi_\infty}$ is defined to be
\[
\pair^{\loc,p}_{\Pi_\infty}:=\Lambda\circ\pair_{(\Whit^{\psi_\infty}(\Pi_\infty),\Whit^{\psi_\infty^{-1}}(\Pi_\infty^\vee),[\cdot,\cdot]_\infty)}^{\coh,p}.
\]
Secondly, in order to define $\pair_{\Pi}^{\glo,p}$ we use the isomorphism \eqref{eq: whipsiHi}: Namely, the pairing $\pair_{\Pi}^{\glo,p}$ is the one which is compatible under \eqref{eq: whipsiHi} with the pairing $\pair^{\Pet,p}$ (defined in \S\ref{sec: pairing}) on
$H^p(\m_G,K_\infty,\Pi\otimes E_\mu)\times H^q(\m_G,K_\infty,\Pi^\vee\otimes E_\mu^\vee)$
and the pairing $[\cdot,\cdot]_f$ on $\Whit^{\psi_f}(\Pi_f)\times\Whit^{\psi_f^{-1}}(\Pi_f^\vee)$.
By \eqref{eq: loctoglob} and our convention of measures we have
\[
\pair_\Pi^{\glo,p}=\abs{D_E}^{n(n+1)/4}\Res_{s=1}L(s,\Pi\times\Pi^\vee)\cdot \pair^{\loc,p}_{\Pi_\infty}.
\]
On the other hand, $\pair_\Pi^{\glo,p}(\strc^{\psi,i}_\Pi,\strc^{\psi^{-1},d-p}_{\Pi^\vee})=\Q(\Pi_f)$. This follows from the definition of $\strc^{\psi,p}_\Pi$ and the fact that (1) $[\cdot,\cdot]_f$ is $\Q(\Pi_f)$-rational with respect to the $\Q(\Pi_f)$-structures on $\Whit^{\psi_f}(\Pi_f)$ and $\Whit^{\psi_f^{-1}}(\Pi_f^\vee)$
and (2) $\pair^{\Pet,p}$ is $\Q(\Pi_f)$-rational with respect to the $\Q(\Pi_f)$-structures on $H^p(\m_G,K_\infty,\Pi\otimes E_\mu)$
and $H^q(\m_G,K_\infty,\Pi^\vee\otimes E^\vee_\mu)$ (see \S\ref{sec: pairing}).
The theorem follows.
\end{proof}

\section{A cohomological interpretation of $\Res_{s=1}L(s,\Pi,\As^{(-1)^{n-1}})$}\label{sect:Asai}

\subsection{Locally symmetric spaces over $F$}
We write $G'_\infty=R_{F/\Q}G(\R)$, where $R_{F/\Q}$ denotes restriction of scalars from $F$ to $\Q$, and denote by $K'_\infty$ the connected component of the identity of the intersection $K_\infty\cap G'_\infty$. It is isomorphic to $\SO(n)^{[F:\Q]}$. We write $A'_G$ for the group of positive reals embedded diagonally in the center of $G'_\infty$. (It will be convenient to distinguish between the isomorphic groups $A_G$ and $A'_G$.) As before, we have $G(\A_F)\cong G(\A_F)^1\times A'_G$ where $G(\A_F)^1:=\{g\in G(\A_F):\abs{\det g}_{\A_F^*}=1\}$. We write $\g'_\infty=\Lie(G'_\infty)$, $\kk'_\infty=\Lie(K'_\infty)$, $\aaa'_G=\Lie(A'_G)$, $\p'=\g'_\infty/\kk'_\infty$ and $\tilde\p'=\p'/\aaa'_{G}$.

Let
$$\SSS_F:=G(F)\bs G(\A_F)^1/K'_\infty$$
be the ``locally symmetric space'' attached to $G(F)$. The closed (non-injective) map $\SSS_F\rightarrow\SSS_E$
gives rise to a map $H_c^q(\SSS_E,\E_\mu)\ra H_c^q(\SSS_F,\E_\mu|_{\SSS_F})$ of $G(\A_{F,f})$-modules.

Finally, let $\epsilon'$ be character on $G(\A_F)$ given by $\varepsilon\circ\det$ if $n$ is even and $1$ if $n$ is odd,
where $\varepsilon$ is the quadratic Hecke character associated to the extension $E/F$ via class field theory.

As before, let $\Pi$ be a cuspidal automorphic representation of $G(\A_E)$ which is cohomological and conjugate self-dual.
Then $\Pi$ is $(G(\A_F),\epsilon')$-distinguished in the sense that
\[
\int_{G(F)\bs G(\A_F)^1}\varphi(h)\epsilon'(h)\ dh
\]
is non-zero for some $\varphi$ in $\Pi$. (Equivalently, (by \cite{MR1344660}) $L(s,\Pi,\As^{(-1)^{n-1}})$ has a pole at $s=1$.)
Indeed, otherwise $L(s,\Pi,\As^{(-1)^n})$ would have a pole, and hence in particular $\Pi_\infty$ would be $(G'_\infty,\chi)$-distinguished
where $\chi=\varepsilon\circ\det$ if $n$ is odd and $\chi=1$ if $n$ is even. However, it is easy to see that this is incompatible
with the description of tempered distinguished representations \cite{panichi}.
See \cite{MR2127943} and \cite{1206.0882}
for the relation between distinction and base change from a unitary group.

\subsection{A archimedean period on cohomology}
Suppose that $\rho$ is a tempered irreducible $(\g_\infty,K_\infty)$-module which is $(G'_\infty,\epsilon')$-distinguished, i.e.,
there exists a non-zero $(G'_\infty,\epsilon')$-equivariant functional $\ell$ on $\rho$.
(Such a functional is unique up to a constant. This follows from \cite{MR2553879}
and the automatic continuity in this context \cite{MR1176208, MR961900}.)
Observe that
\[
t=\frac{n(n+1)}{2}[F:\Q]-1=\dim_\R\SSS_F=\dim_\R\tilde\p',
\]
where $t$ is the highest degree for which $H^t(\m_G,K_\infty,\rho\otimes E_\mu)$ can be non-zero for a generic irreducible (essentially unitary) representation $\rho$. Moreover, $H^t(\m_G,K_\infty,\rho\otimes E_\mu)$ is one-dimensional.

Let $V_\lambda$ be a highest weight representation of $\GL_n(\R)$ with parameter $\lambda=(\lambda_1,\dots,\lambda_n)$
and let $\lambda^\vee=(-\lambda_n,\dots,-\lambda_1)$. Let $(\cdot,\cdot)_\lambda$ be the standard pairing
on $V_\lambda\times V_{\lambda^\vee}$.
Since by assumption $E_\mu$ is conjugate self-dual, we can define a $G'_\infty$-invariant form $\ell_\mu:E_\mu\rightarrow\C$ by taking
the tensor product of the pairings above over all archimedean places of $E$.

We define a functional
$$\per_{(\rho,\ell)}^{\coh,t}: H^t(\m_G,K_\infty,\rho\otimes E_\mu)\ra (\wedge^t\tilde\p')^*$$ as follows. Suppose that $\tilde\omega\in \Hom_{K_\infty}(\wedge^t\tilde\p,\rho\otimes E_\mu)$ represents $\omega\in H^t(\m_G,K_\infty,\rho\otimes E_\mu)$. We compose $\tilde\omega$ with the embedding $\wedge^t\tilde\p'\hookrightarrow\wedge^t\tilde\p$ and with $\ell\otimes\ell_\mu$ to get an element of $\per_{(\rho,\ell)}^{\coh,t}(\omega)\in(\wedge^t\tilde\p')^*$. We will make the following assumption:
\begin{hyp} \label{as: nonzero}
$\per_{(\rho,\ell)}^{\coh,t}$ is non-zero.
\end{hyp}

Hopefully, this will be proved in the near future using the method of Binyong Sun (cf. \cite{1307.5357, 1111.2636}).

Next, we fix a $\Q$-valued Haar measure $\gamma'$ on $G(\A_{F,f})$: As in \S\ref{sec: pairing} we use $\gamma'$ to define the normalized integrals
\[
H^t_c(\SSS_F,\underline{\C})\xrightarrow{\int'_{\SSS_F,\epsilon'}}\C,
\]
except that now we take the cup product with the class $[\epsilon']\in H^0(\SSS_F,\underline{\C})$ represented by $\epsilon'$ before integrating.
By composing $\int'_{\SSS_F,\epsilon'}$ with the map $H^t_c(\SSS_F,\E_\mu|_{\SSS_F})\rightarrow H^t_c(\SSS_F,\underline{\C})$
induced from $\ell_\mu$ and the map $H_c^t(\SSS_E,\E_\mu)\ra H_c^t(\SSS_F,\E_\mu|_{\SSS_F})$, we get a period map $H^t_c(\SSS_E,\E_\mu)\rightarrow\C$.
As before, this map is $\Aut(\C)$-equivariant. Composing with $\Delta^t$ we finally obtain a linear form
\[
\per^{\gper,t}: H^t(\m_G,K_\infty,\Pi\otimes E_\mu)\ra\C.
\]
It coincides with the volume of $\SSS_F$ with respect to the complex-valued measure
$\per_{(\Pi_\infty,\ell_{\aut})}^{\coh,t}\otimes\gamma'$ of $G(\A_F)^1/K'_\infty\cong G'_\infty/A'_G K'_\infty\times G(\A_{F,f})$
where
\[
\ell_{\aut}(\varphi)=\int_{G(F)\bs G(\A_F)^1}\varphi(h)\epsilon'(h)\ dh
\]
and we identify an element of $(\wedge^t\tilde\p')^*$ with an invariant measures on $ G'_\infty/A'_G K'_\infty$.
In particular,  $\per^{\gper,t}$ is non-zero if we assume Hypothesis \ref{as: nonzero}.

\subsection{Measures over $F$} \label{sec: measures_F}

We now fix some measures. If $v$ is a non-archimedean place of $F$, we take the Haar measure on $F_v$ which gives volume one to the integers $\OO_v$.
On $\R$ we take the Lebesgue measure.
This gives rise to Tamagawa measures on the local groups $G(F_v)$, $\mira(F_v)$ and $U(F_v)$ by taking the standard
gauge form as in \S\ref{sec: measures}.
Thus, if $v$ is non-archimedean then the volume of $G(\OO_v)$ is $\Delta_{G,v}^{-1}$ where
$\Delta_{G',v}=\prod_{j=1}^nL(j,\one_{F_v^*})$.
On $G(\A_F)$ and $G(\A_{F,f})$ we will take the measure
\[
\prod_v\Delta_{G',v}\ dg_v
\]
where $v$ ranges over all places (resp., all finite places).
The measure on $A'_G$ will be determined by the isomorphism $\abs{\det}_{\A_F^*}:A'_G\rightarrow\R_+$
and the measure $dx/x$ on $\R_+$ where $dx$ is the Lebesgue measure.
The isomorphism $G(\A_F)\cong A'_G\times G(\A_F)^1$ gives rise to a measure on $G(\A_F)^1$.
Then
\[
\vol(G(F)\bs G(\A_F)^1)=\abs{D_F}^{n^2/2}\Res_{s=1}\prod_{j=1}^n\zeta_F^*(s+j-1)
\]
where $D_F$ is the discriminant of $F$ and $\zeta_F^*(s)$ is the completed Dedekind zeta function of $F$.

Let $\xi'\in(\wedge^t\tilde\p')^*$ correspond to the invariant measure on $G'_\infty/A'_G K'_\infty$
obtained by the push-forward of the Haar measure on $G'_\infty/A'_G$ chosen above.
Let $\Lambda_0'\in\wedge^t\tilde\p'$ be the element such that $\xi'(\Lambda_0')=1$.

\subsection{The Whittaker realization of $\ell_{\aut}$}
Recall that $\psi$ was a fixed character of $E\bs\A_E$.
We assume from now on that the restriction of $\psi$ to $F\bs\A_F$ is trivial.

Given a finite set of places $S$ of $F$ and an irreducible generic unitarizable $(G(F_S),\epsilon')$-distinguished representation $\pi_S$ of $G(E_S)$ with Whittaker
model $\Whit^{\psi_S}(\pi_S)$ define
\[
 \ell_S(W):=\frac{\Delta_{G',S}}{L(1,\pi_S,\As^{(-1)^{n-1}})}\cdot
 \int_{U(F_S)\bs\mira(F_S)}W(h)\epsilon'(h)\ dh, \ \ W\in\Whit^{\psi}(\pi_S)
\]
where $\Delta_{G',S}=\prod_{v\in S}\Delta_{G',v}$.
The integral converges and defines a $(G(F_S),\epsilon')$-equivariant form on $\Whit^{\psi_S}(\pi_S)$ (\cite{MR2787356, 1212.6436}).

If $v$ is a non-archimedean place of $F$, then by the same argument as in \S\ref{eq: equiRS} we have
\[
L(s,{}^\sigma\Pi_v,\As^{(-1)^{n-1}})=L(s,\Pi_v,\As^{(-1)^{n-1}})^\sigma
\]
(where on the right-hand side, $\sigma$ acts on $\C(q^{-s})$ in the obvious way). In particular,
\begin{equation} \label{eq: L1as}
L(1,{}^\sigma\Pi_v,\As^{(-1)^{n-1}})=\sigma(L(1,\Pi_v,\As^{(-1)^{n-1}})).
\end{equation}
Thus, if $S$ consists only of non-archimedean places, then $\ell_S(W)$ is $\Aut(\C)$-equivariant, i.e.,
\[
  \ell_S({}^\sigma W)=\sigma(\ell_S(W)).
\]
Indeed, by uniqueness (\cite{MR1111204}), it suffices to check this relation when the restriction of $W$ to $\mira(E_S)$ is compactly supported modulo $U(E_S)$,
in which case the integral reduces to a finite sum and the assertion follows from \eqref{eq: L1as} and the
rationality of the measure on $U(F_S)\bs\mira(F_S)$.

With our choice of measures, for any cuspidal automorphic form $\varphi$ in the space of $\Pi$ and $\whit^\psi_\varphi=\whit^\psi(\varphi)$, we have
\begin{equation} \label{eq: loctoglobAs}
\ell_{\aut}(\varphi)=\abs{D_F}^{n(n+1)/4}\Res_{s=1}L(s,\Pi,\As^{(-1)^{n-1}})\cdot \ell_S(\whit^\psi_\varphi)
\end{equation}
provided that $S$ is a sufficiently large finite set of places of $F$. (Cf. Gelbart--Jacquet--Rogawski, \cite{MR1882037}, pp.\ 184--185 or Zhang \cite{MR3164988}, Sect.\ 3.2.)
Note that $\ell_S(\whit^\psi_\varphi)$ is unchanged by enlarging $S$ because of the extra factor $\Delta_{G',S}$ in the numerator.

As before we write $\ell_f(W)=\ell_S(W)$ for $W\in\Whit^{\psi_f}(\Pi_f)$ where $S$ is any sufficiently large set of places.

\subsection{A relation between the top Whittaker \regulator\ and $\Res_{s=1}L(1,\Pi,\As^{(-1)^{n-1}})$}
We may now prove our first main theorem on the Asai $L$-function.

\begin{thm}\label{thm:Asai}
Let $\Pi$ be a conjugate self-dual, cuspidal automorphic representation of $G(\A_E)=\GL_n(\A_E)$, which is cohomological with respect to an irreducible,
finite-dimensional, algebraic representation $E_\mu$. Assume Hypothesis \ref{as: nonzero}.
Then, $\left(\abs{D_F}^{n(n+1)/4}\Res_{s=1}L(s,\Pi,\As^{(-1)^{n-1}})\right)^{-1}$ spans the one-dimensional $\Q(\Pi_f)$-vector space
\[
\Lambda'\circ\per_{(\Whit^{\psi_\infty}(\Pi_\infty),\ell_\infty)}^{\coh,t}(\strc^{\psi,t}_\Pi)\subset\C
\]
where $\Lambda':(\wedge^t\tilde\p')^*\rightarrow\C$ is the evaluation at the element $\Lambda_0'\in\wedge^t\tilde\p'$ defined above.
\end{thm}

\begin{proof}
Arguing as in the proof of Thm.\ \ref{thm:Whittaker}, the result follows using relation \eqref{eq: loctoglobAs}, which compares the
global period-map $\per^{\gper,t}$ with the local period-map $\Lambda'\circ\per_{(\Whit^{\psi_\infty}(\Pi_\infty),\ell_\infty)}^{\coh,t}$,
and the fact that $\ell_f({}^\sigma W)=\sigma(\ell_f(W))$ for all $\sigma\in\Aut(\C)$.
\end{proof}

\section{A relation between the bottom Whittaker \regulator\ and $L(1,\Pi,\As^{(-1)^{n}})$}
\subsection{}
We will now put the contents of \S \ref{sect:RS} and \S \ref{sect:Asai} together, in order to obtain a rationality result for $L(1,\Pi,\As^{(-1)^{n}})$. To that end, recall the pairing $\pair_{(\Whit^{\psi_\infty}(\Pi_\infty),\Whit^{\psi_\infty^{-1}}(\Pi_\infty^\vee),[\cdot,\cdot]_\infty)}^{\coh,i}$ from \S \ref{sect:pairing_gK}. We use it to identify $\Lambda'\circ\per_{(\Whit^{\psi_\infty}(\Pi_\infty),\ell_\infty)}^{\coh,t}$ as an element of
$$H^b(\m_G,K_\infty,\Whit^{\psi_\infty^{-1}}(\Pi_\infty^\vee)\otimes E_\mu^\vee)\otimes\wedge^d\tilde\p,$$
where $b=\frac{n(n-1)}{2}[F:\Q]$. (Recall that $b+t=d$.) We write this element as $\coper_{(\Whit^{\psi_\infty}(\Pi_\infty),\ell_\infty)}^{\coh,b}\otimes\Lambda_0$, where
\[
\coper_{(\Whit^{\psi_\infty}(\Pi_\infty),\ell_\infty)}^{\coh,b}\in H^b(\m_G,K_\infty,\Whit^{\psi_\infty^{-1}}(\Pi_\infty^\vee)\otimes E_\mu^\vee).
\]
We may now prove our second main theorem on the Asai $L$-function.

\begin{thm} \label{thm: crit}
Let $\Pi$ be a conjugate self-dual, cuspidal automorphic representation of $G(\A_E)=\GL_n(\A_E)$, which is cohomological with respect to an irreducible,
finite-dimensional, algebraic representation $E_\mu$. Assume Hypothesis \ref{as: nonzero}. Then,
\[
\left(\abs{D_E/D_F}^{n(n+1)/4}L(1,\Pi,\As^{(-1)^n})\right)^{-1} \cdot \coper_{(\Whit^{\psi_\infty}(\Pi_\infty),\ell_\infty)}^{\coh,b}
\]
spans the one-dimensional $\Q(\Pi_f)$-vector subspace $\strc^{\psi^{-1},b}_{\Pi^\vee}$ of
$H^b(\m_G,K_\infty,\Whit^{\psi_\infty^{-1}}(\Pi_\infty^\vee)\otimes E_\mu^\vee)$.
\end{thm}

\begin{proof}
The theorem follows readily from Thm.\ \ref{thm:Whittaker} and Thm.\ \ref{thm:Asai} and the relation
$L(s,\Pi \times\Pi^\tau) = L(s,\Pi,\As^{(-1)^{n-1}})\cdot L(s,\Pi,\As^{(-1)^{n}})$.
\end{proof}

\begin{rem}
Theorem \ref{thm: crit} generalizes a result of the first two named authors, see \cite{1308.5090} Thm.\ 6.22.
\end{rem}

\bigskip

\bibliographystyle{amsalpha}

\def\cprime{$'$}
\providecommand{\bysame}{\leavevmode\hbox to3em{\hrulefill}\thinspace}
\providecommand{\MR}{\relax\ifhmode\unskip\space\fi MR }
\providecommand{\MRhref}[2]{%
  \href{http://www.ams.org/mathscinet-getitem?mr=#1}{#2}
}
\providecommand{\href}[2]{#2}

\end{document}